\documentclass{article}
\usepackage[margin=1.25in]{geometry}
\usepackage{amsmath}
\usepackage{amssymb}
\usepackage{amsthm}
\usepackage[all]{xy}
\usepackage{hyperref}
\usepackage{mathtools}
\usepackage{enumerate}
\usepackage{dsfont}
\usepackage{mathrsfs}
\usepackage{tikz}

\newcommand{\Z}{\mathbb{Z}}
\newcommand{\R}{\mathbb{R}}

\newcommand{\cm}{{\mathfrak{c}}}

\newcommand{\im}{{\mathrm{im}\;}}

\newtheorem{theorem}{Theorem}
\newtheorem{lemma}{Lemma}
\newtheorem{claim}{Claim}
\newtheorem{proposition}{Proposition}
\newtheorem{corollary}{Corollary}

\theoremstyle{definition}

\newtheorem{example}{Example}
\newtheorem{remark}{Remark}

\title{On the second homotopy group of the classifying space for commutativity in Lie groups}
\author{Bernardo Villarreal}
\date{\today}

\begin{document}

\maketitle


\begin{abstract}
In this note we show that the second homotopy group of $B(2,G)$, the classifying space for commutativity for a compact Lie group $G$, contains a direct summand isomorphic to $\pi_1(G)\oplus\pi_1([G,G])$, where $[G,G]$ is the commutator subgroup of $G$. It follows from a similar statement for $E(2,G)$, the homotopy fiber of the canonical inclusion $B(2,G)\hookrightarrow BG$. As a consequence of our main result we obtain that if $E(2,G)$ is 2-connected, then $[G,G]$ is simply-connected. This last result completes how the higher connectivity of $E(2,G)$ resembles the higher connectivity of $[G,G]$ for a compact Lie group $G$.
\end{abstract}

\section*{Introduction}

Let $G$ be a Lie group. A. Adem, F. Cohen and E. Torres-Giese \cite{AdemCohenTorres12} studied a filtration of the classifying space $BG$ associated to the descending central series of a group. The space $B(2,G)$ sitting in the first term of this filtration arises by assembling together the spaces of commuting tuples in $G$. The homotopy fiber of the inclusion $B(2,G)\hookrightarrow BG$, denoted $E(2,G)$ can be thought of as the difference between $B(2,G)$ and $BG$, and in some sense measures how far is $G$ from being abelian. In this note we study the second homotopy group $\pi_2(E(2,G))$ for a compact Lie group $G$. The homotopy classes in $\pi_2(E(2,G))$ for the orthogonal matrix groups $G=\mathrm{O}(n)$,  were previously used in \cite{RamrasVillarreal19} to produce non-standard classes in the reduced commutative orthogonal $K$-theory of closed connected surfaces. This cohomology theory is a variant of orthogonal $K$-theory (as defined in \cite{AdemGomez15} and further studied in \cite{AdGoLiTi15}), whose classes are represented by orthogonal vector bundles equipped with transition functions that point-wise commute. In general for a Lie group $G$, a principal $G$-bundle equipped with such transition functions is said to have a transitionally commutative structure.

Let $[G,G]$ denote the commutator subgroup of $G$. To prove our main result we combine the commutator map $\mathfrak{c}\colon E(2,G)\to B[G,G]$ introduced in \cite{AntolinGritschacherVillarreal21} with the techniques developed in \cite{RamrasVillarreal19} of producing new classifying maps from old ones by inverting cocycles that point-wise commute. For instance, in the case of the special orthogonal groups $\mathrm{SO}(n)$, this procedure can be translated into $\mathfrak{c}_*\colon \pi_2(E(2,\mathrm{SO}(n)))\to \pi_2(B\mathrm{SO}(n))$ being an isomorphism, for every $n\geq 3$. Our main result is to extend these ideas to any compact Lie group.

\begin{theorem}\label{MainTheorem1}
Let $G$ be a compact Lie group. Then:
\begin{enumerate}
\item $\pi_2(E(2,G))$ contains a direct summand isomorphic to $\pi_2(B[G,G])$.
\item If $G$ is connected, then $\mathfrak{c}_*\colon\pi_2(E(2,G))\xrightarrow{\cong} \pi_2(B[G,G])$.
\end{enumerate}
\end{theorem}

As an application of part 2 of Theorem \ref{MainTheorem1}, in Corollary \ref{cor-tcstdim3} we show that when $G$ is a connected compact Lie group, every principal $G$-bundle over a CW complex of dimension $\leq 3$, with a reduction to $[G,G]$, possesses a transitionally commutative structure.

In a slightly different context, when $G$ is discrete, $E(2,G)$ is homotopy equivalent to the coset poset of abelian subgroups of $G$. Coset posets of families of subgroups were studied by Abels and Holz \cite{AbelsHolz90}, who were interested in their higher connectivity and the algebraic implications over $G$. For instance, the coset poset of the family of abelian subgroups is simply connected if and only if $G$ is abelian (see \cite{Okay15}). 

When $G$ is a compact Lie group, we can also treat $E(2,G)$ as the Lie group analogue of the coset poset of abelian subgroups. In this scenario, in \cite[Theorem 1]{AntolinGritschacherVillarreal21} it is shown that $E(2,G)$ is contractible if and only if $G$ is abelian, and it further asserts that the latter holds if and only if $E(2,G)$ is 4-connected.  As for compact Lie groups, $G$ abelian is equivalent to $[G,G]$ 3-connected, our motivation is then to understand the relation between the higher connectivity of $E(2,G)$ and the higher connectivity of $[G,G]$.

 As a consequence of part 1 of our main result we obtain: 

\begin{corollary}\label{cor-e2g2conn}
Let $G$ be a compact Lie group. If $E(2,G)$ is $2$-connected, then $[G,G]$ is $1$-connected.
\end{corollary}

By \cite[Proposition 9]{AntolinGritschacherVillarreal21}, the vanishing of the fundamental group $\pi_1(E(2,G))$ implies that $[G,G]$ is connected. Combining this with Corollary \ref{cor-e2g2conn}, \cite[Theorem 1]{AntolinGritschacherVillarreal21} and the fact that $\pi_2([G,G])$ is always trivial, yields for a compact Lie group, that the higher connectivity of $E(2,G)$ completely determines how highly connected is the commutator subgroup. To spell this out precisely: if $E(2,G)$ is $n$-connected, then $[G,G]$ is $(n-1)$-connected for $n=1,2$ and $4$. In this context we then obtain a full interpretation of how does $E(2,G)$ measures how far is $G$ from being abelian. 

Perhaps it is worth mentioning that the analogy does not work with non-compact Lie groups. For example, consider $\mathrm{SL}(2,\R)$ with its maximal compact subgroup $\mathrm{SO}(2)$, which is abelian. By \cite[Theorem 3.1]{AdemGomez15} there is a homotopy equivalence $E(2,\mathrm{SL}(2,\R))\simeq E\mathrm{SO}(2)$ making $E(2,\mathrm{SL}(2,\R))$ contractible, but $[\mathrm{SL}(2,\R),\mathrm{SL}(2,\R)]=\mathrm{SL}(2,\R)$, which is not simply-connected. 

We finish the introduction with the following observation. In \cite[Question {21}]{AntolinGritschacherVillarreal21} the authors asked if the looped commutator map $\Omega\mathfrak{c}\colon\Omega E(2,G)\to[G,G]$ splits, up to homotopy, for a compact Lie group $G$. Our techniques allow us to answer in the affirmative for a family of extensions of finite groups by tori (see Corollary \ref{cor-omegacsplit}), which for instance contains every normalizer $N(T)$ of a maximal torus $T$ in a semisimple connected compact Lie group. 

The paper is organized as follows. In Section \ref{sec-pre} we briefly recall simplicial models of $B(2,G)$ and $E(2,G)$, and argue that, for a connected compact Lie group, the homotopy type of $E(2,G)$ only depends on the semisimple part of $G$. In Section \ref{Section Gcon}, for semisimple connected compact Lie groups, we prove the existence of an isomorphism $\pi_2(E(2,G))\cong \pi_1(G)$ by analyzing the Moore complex associated to the simplicial abelian group $(n\mapsto H_0(C_n(G);\Z))$, where $C_n(G)$ is the space of commuting $n$-tuples in $G$. In Section \ref{Section extens} we study the fundamental group of the commutator subgroup of extensions of finite groups by a torus, and describe its generators via the commutator map. In Section \ref{SectionProofMainTh} we prove Theorem \ref{MainTheorem1}, and the applications mentioned above.

\subsection*{Acknowledgements}
The author would like to thank O. Antol\'in-Camarena for pointing out Lemma \ref{lemma-augmidquo}, and S. Gritschacher for useful conversations, particularly Example \ref{prop-pi2cm-surj-conn}. The author acknowledges support from Universidad Nacional Aut\'onoma de M\'exico (UNAM) in the program ``Becas de Posdoc DGAPA.'' 

\paragraph{Some conventions and notation.} Throughout this paper we will use the following conventions:
\begin{itemize}
\item For a pair of elements $x,y$ in a group $G$, the commutator $[x,y]$ will mean $x^{-1}y^{-1}xy$.
\item For a compact Lie group $G$, $G_0$ will denote the connected component of the identity, and $Z(G)$ will denote the center of $G$. 
\end{itemize}

\section{A few preliminaries}\label{sec-pre}

Let $G$ be Lie group. Consider the space of commuting $k$-tuples 
\[C_k(G)=\{(g_1,...,g_k)\in G^k:[g_i,g_k]=1\}\,,\]
 with its subspace inclusion $i_k\colon C_k(G)\hookrightarrow G^k$. A a simplicial model $B_\bullet G$ of the classifying space $BG$ with $B_kG=G^k$, is the nerve of $G$, where $G$ is taken as a topological category with a single object. In this model the simplicial structure of $B_\bullet G$ is compatible with every inclusion $i_k$, so that defining $C_0(G)=\text{pt}$ the family $\{C_k(G)\}_{k\geq 0}$ assembles into a simplicial space $C_\bullet(G)$ where the induced inclusion $i_\bullet\colon C_\bullet(G)\to B_\bullet G$ becomes a simplicial map. Taking geometric realization the classifying space for commutativity is defined as
\[B(2,G):=|C_\bullet(G)|\,,\]
and is equipped with a canonical inclusion $i:=|i_\bullet|\colon B(2,G)\to BG$. (See \cite{AdemCohenTorres12} for more details). The space $E(2,G)$ is defined as the homotopy fiber of $i$, that is, the pullback of the universal principal $G$-bundle $EG\to BG$ along $i$. A simplicial model for $E(2,G)$ is as follows. For every $k\geq 0$ let
\[E_k(2,G)=\{(g_0,...,g_k): \langle g_0^{-1}g_1,...,g_{k-1}^{-1}g_k\rangle\;\text{is abelian}\}\,,\]
and $E(2,G)=|E_\bullet(2,G)|$. It comes with a projection $p\colon E(2,G)\to B(2,G)$, given as the realization of the simplicial map $E_\bullet(2,G)\to C_\bullet(G)$ where for every $k\geq 1$, $(g_0,...,,g_k)\mapsto (g_0^{-1}g_1,...,g_{k-1}^{-1}g_k)$. See \cite[Section 2]{AntolinGritschacherVillarreal21} for more details.

We record 2 properties of these spaces. First, if $G$ is abelian clearly $B(2,G)=BG$, and hence $E(2,G)=EG$ which is contractible. Secondly, by construction $B(2,G)$ is connected, and the first term in the skeletal filtration of $E(2,G)$ is homotopy equivalent to the suspension $\Sigma (G\wedge G)$, so that $E(2,G)$ is connected, as well.

Now we will argue that the problem of studying the homotopy type of $E(2,G)$ for connected compact Lie groups $G$ can be reduced to semisimple connected compact Lie groups (i.e. its Lie algebra is a direct sum of simple Lie algebras). 

\begin{lemma}\label{lem-e2centralproduct}
Suppose $H$ and $K$ are Lie groups or discrete groups, where $H$ is non-trivial and abelian. Suppose further that $Z$ is a common central subgroup, and consider the central product $G=H\times_{Z} K$ where $Z\subset H\times K$ via $z\mapsto (z,z^{-1})$. Then the inclusion $K\hookrightarrow G$ induces a homotopy equivalence
\[E(2,K)\xrightarrow{\simeq}E(2,G)\,.\]
\end{lemma}

\begin{proof} 
Let $\pi\colon H\times K\to G$ be the quotient map. We claim that the commutative diagram
\begin{align}
\xymatrix{
H^n\times C_n(K)\ar[d]^{\pi^n|}\ar@{^(->}[r]&(H\times K)^n\ar^{\pi^n}[d]\\
C_n(G)\ar@{^(->}[r]&G^n
\,,}\label{cen-decomp-hom}
\end{align}
is a pullback square. Notice that if $(\pi(h_1,k_1),...,\pi(h_n,k_n))\in G^n$ is a commuting $n$-tuple, then $[(h_i,k_i),(h_j,k_j)]\in Z$ for every $1\leq i,j\leq n$, but since the $h_i$'s pairwise commute, the only possible value of these commutators is $(1,1)\in Z$. In particular $[k_i,k_j]=1$ for every $1\leq i,j\leq n$ and the claim now follows. The restriction $\pi^n|$ is then a principal $Z^n$-bundle since $\pi^n$ is a principal bundle, as well. 

Upon taking geometric realization of (\ref{cen-decomp-hom}) we again obtain a pullback square and hence a principal $BZ$-bundle $BH\times B(2,K)\to  B(2,G)$. Now the inclusion $B(2,G)\to BG$ being natural induces a map of fiber sequences
\[
\xymatrix{BZ\ar[r]\ar@{=}[d]&BH\times B(2,K)\ar[d]\ar[r]&B(2,G)\ar[d]\\
BZ\ar[r]&BH\times BK\ar[r]&BG\,,}
\]
and taking vertical homotopy fibers gives the claimed homotopy equivalence.  
\end{proof}

Let $G$ be a compact Lie group and $H,K\subset G$ closed subgroups. Let us briefly recall that the commutator subgroup $[H,K]\subset G$ is the topological closure of the algebraic commutator subgroup $[H,K]^{\text{alg}}$, but since $G$ is compact, $[H,K]=[H,K]^{\text{alg}}$.
 

\begin{proposition}\label{lem-e2semisimple}
Let $G$ be a connected compact Lie group. Then the inclusion $[G,G]\hookrightarrow G$ induces a homotopy equivalence $E(2,[G,G])\xrightarrow{\simeq} E(2,G)$.
\end{proposition}
\begin{proof}
Let $Z(G)_0$ denote the path connected component of the identity of the center of $G$. A well known decomposition in the theory of connected compact Lie groups (e.g. \cite[Theorem 4.29]{Knapp}) is $G=Z(G)_0[G,G]$. In this decomposition the commutator subgroup $[G,G]$ is the maximal semisimple subgroup of $G$. In particular the product map $\pi\colon Z(G)_0\times [G,G]\to G$ is a covering homomorphism with kernel $Z(G)_0\cap [G,G]=Z$. Notice that $G$ can then be decomposed as the central product $Z(G)_0\times_{Z}[G,G]$, where the central subgroup $Z$ is included in $Z(G)_0\times[G,G]$ via $z\mapsto(z,z^{-1})$. The result now follows from Lemma \ref{lem-e2centralproduct}.
\end{proof}

\section{Calculation of $\pi_2(E(2,G))$ for connected compact Lie groups}\label{Section Gcon}

Recall that when $G$ is a connected compact Lie group, $[G,G]$ is a semisimple connected compact Lie group, and if further $G$ is semisimple, then the commutator subgroup $[G,G]=G$. In light of Proposition \ref{lem-e2semisimple}, to study $E(2,G)$, hereon we may assume that $G$ is a semisimple connected compact Lie group, if necessary. 

The next step is to understand the Moore complex associated to the simplicial abelian group $n\mapsto H_0(C_n(G);\Z)$. To do this we will need the following two lemmas.

\begin{lemma}\label{lemma-augmidquo}
Let $A$ be a discrete group and let $W$ be the augmentation ideal of the group ring $\Z A$. Then $W_A$ the module of $A$-coinvariants of $W$ is isomorphic as an abelian group to the abelianization of $A$, that is, $W_A=W/\langle w-aw:a\in A\rangle\cong A/[A,A]$.
\end{lemma}

\begin{proof}

Consider the short exact sequence $0\to W\to \Z A\to \Z\to 0$ and the induced long exact sequence in homology
\[H_1(A;\Z A)\to H_1(A;\Z)\to H_0(A;W)\to H_0(A;\Z A)\to H_0(A;\Z)\;.\]
Recall that $H_i(A;\Z A)\cong H_i(EA;\Z)$ for every $i\geq 0$, where $EA$ is the contractible total space of the universal principal $A$-bundle $EA\to BA$. Then notice that $H_0(A;\Z A)\to H_0(A;\Z)$ is an isomorphism since it can be realized as $H_0(-;\Z)$ of $EA\to BA$. From the above long exact sequence we thus have an isomorphism $A/[A,A]=H_1(A;\Z)\xrightarrow{\cong}H_0(A;W)= W_A$.
\end{proof}

Let $K$ be a subgroup $K\subset Z(G)$, where $Z(G)$ is the center of $G$. We denote the space of almost commuting $n$-tuples relative to $K$ by $B_n(G,K)=\{(g_1,...,g_n): [g_i,g_j]\in K\}$.

\begin{lemma}\label{lemma-almcomtrip}
Let $G$ be a semisimple simply-connected compact Lie group, and $K\subset Z(G)$. Then for every pair of elements $c_1,c_2\in K$, there is a triple $(g_1,g_2,g_3)\in B_3(G,K)$ such that $[g_2,g_3]=c_{1}$, $[g_1,g_3]=c_{2}$ and $[g_1,g_2]=1$. 
\end{lemma}

\begin{proof}
Suppose $K$ is cyclic and let $c$ be a generator. We can find a pair $(g,h)\in G^2$ such that $[g,h]=c$, for example by \cite[Appendix A]{KacSmilga} which shows how to construct Heisenberg pairs in $G$. Then $(1,c^i,c^j)$ can be realized for instance by the associated triple of commutators of $(g^j,g^i,h)\in G^3$.

Among simple Lie groups, the only ones that have non-cyclic center are the spin groups $\mathrm{Spin}(4k)$, for $k\geq2$, whose center is $\Z/2\times \Z/2$. Consider $\mathrm{Spin}(4)\cong \mathrm{SU}(2)\times\mathrm{SU}(2)$. One can readily check that every triple in $Z(\mathrm{SU}(2)\times\mathrm{SU}(2))=\Z/2\times\Z/2$ can be realized by almost commuting triples of $\mathrm{SU}(2)\times\mathrm{SU}(2)$. Then the diagonal block matrix inclusion $\mathrm{SO}(4)\hookrightarrow \mathrm{SO}(4k)$ induces a map $\delta\colon \mathrm{Spin}(4)\to \mathrm{Spin}(4k)$, which is an isomorphism restricted to the center. Thus any triple of $Z(\mathrm{Spin}(4k))$ can be realized by a triple of commutators of elements in the image of $\delta$.  

Without loss of generality suppose $G$ is a product of two simple groups $G=G_1\times G_2$. Pick $(a_1,b_1),(a_2,b_2)\in Z(G)=Z(G_1)\times Z(G_2)$. The previous argument gives us triples $(g_1,g_2,g_3)\in B_3(G_1,Z(G_1))$ and $(h_1,h_2,h_3)\in B_3(G_2,Z(G_2))$ that realize $(1,a_1,a_2)$ and $(1,b_1,b_2)$, respectively. Then $((g_1,h_1),(g_2,h_2),(g_3,h_3))\in B_3(G,Z(G))$ realizes $((1,1),(a_1,b_1),(a_2,b_2))\in Z(G)^3$.
\end{proof}

Let $(D,\partial)$ be the Moore complex associated to the simplicial group $H_0(C_\bullet(G);\Z)$. In \cite[Lemma 4.2]{RamrasVillarreal19} it is shown that in our context $H_2(B(2,G);\Z)\cong H_1(G;\Z)\oplus H_2(D,\partial)$. Since both $B(2,G)$ and $E(2,G)$ are simply connected when $G$ is connected (e.g. \cite[Lemma 4.3]{AdGoLiTi15}), and $\pi_2(B(2,G))\cong \pi_2(BG)\oplus\pi_2(E(2,G))$ (e.g. \cite[Theorem 6.3]{AdemCohenTorres12}), we conclude that \[\pi_2(E(2,G))\cong H_2(D,\partial)\,.\]

\begin{proposition}\label{prop-pi2E2g}
Let $G$ be a semisimple connected compact Lie group. Then $\pi_2(E(2,G))\cong \pi_1(G)$. 
\end{proposition}

\begin{proof}
We will show that $H_2(D,\partial)\cong \pi_1(G)$. We are interested in the following portion of the chain complex $(D,\partial)$:
\[H_0(C_3(G);\Z)\xrightarrow{\partial_3}H_0(C_2(G);\Z)\xrightarrow{\partial_2}H_0(G;\Z)\xrightarrow{}0\,.\]

 A result of Kac and Smilga \cite[Remark A1]{KacSmilga} says that the path-connected components of $C_2(G)$ are in bijection with the elements of $\pi_1(G)$, so that under this correspondence we can identify $H_0(C_2(G);\Z)$ with the group ring $\Z\pi_1(G)$. As $G$ is connected, the differential $\partial_2$ is the augmentation map with respect to $\Z\pi_1(G)$. Hence $\ker \partial_2$ is the augmentation ideal. 

To analyze $\partial_3$, let us consider $\widetilde{G}$, the universal cover of $G$. Then $\widetilde{G}/Z(\widetilde{G})=G$, so that for every element in $C_k(G)$ there is a lift to $B_k(\widetilde{G},Z(\widetilde{G}))$ with respect to the projection $\widetilde{G}\to G$. Since $\widetilde{G}$ is a semisimple simply-connected compact Lie group, for every triple of the form $(1,\eta_1,\eta_2)\in\pi_1(G)^3\cong Z(\widetilde{G})^3$, Lemma \ref{lemma-almcomtrip} asserts that we can find a triple $(g_1,g_2,g_3)\in C_3(G)$ and a lift $(\tilde{g_1},\tilde{g_2},\tilde{g_3})\in B_3(\tilde{G};Z(\widetilde{G}))$ such that its associated triple of commutators is $([\tilde{g_1},\tilde{g_2}]=1,[\tilde{g_2},\tilde{g_3}]=\eta_1,[\tilde{g_1},\tilde{g_3}]=\eta_2)$. 

Now, let us label a connected component of $C_3(G)$ with associated triple of commutators $(1,\eta_1,\eta_2)$ by $c_{(1,\eta_1,\eta_2)}$. To understand what the value of $\partial_3(c_{(1,\eta_1,\eta_2)})$ is, we need to analyse what the commutators of the pairs $d_i(\tilde{g_1},\tilde{g_2},\tilde{g}_3)$ are, where $d_i\colon C_3(G)\to C_2(G)$ are the face maps and $0\leq i\leq 3$. Since the commutators are central, we can easily verify that for each $i$ we have: 

\begin{itemize}
\item $i=0$. $[\tilde{g_2},\tilde{g}_3]=\eta_1$
\item $i=1$. $[\tilde{g_1}\tilde{g_2},\tilde{g}_3]=\eta_2\eta_1$
 \item $i=2$. $[\tilde{g_1},\tilde{g_2}\tilde{g}_3]=\eta_2$
 \item $i=3$. $[\tilde{g_1},\tilde{g}_2]=1$
\end{itemize}
We conclude that the image of these generators under $\partial_3$ is given by
\[\partial_3(c_{(1,\eta_1,\eta_2)})=\eta_1-\eta_2\eta_1+\eta_2-1=\eta_1(1-\eta_2)-(1-\eta_2)\,.\]
As the elements $1-\eta_2$ form a basis of $\ker\partial_2$, $\ker\partial_2/\mathrm{im}\;\partial_3$ is precisely the group of coinvariants $W_{\pi_1(G)}$. The proposition now follows from Lemma \ref{lemma-augmidquo}.
\end{proof}

\begin{remark}\label{rem-pi2e2gconn}
Proposition \ref{lem-e2semisimple} and Proposition \ref{prop-pi2E2g} then yield an isomorphism $\pi_2(E(2,G))\cong\pi_1([G,G])$, for every connected compact Lie group $G$.
\end{remark}

\begin{remark}
When $G$ is not connected, there is no isomorphism between $\pi_2(E(2,G))$ and $\pi_2(B[G,G])$. For instance consider $G=\mathrm{O}(2)$, where $[G,G]=\mathrm{SO}(2)$. In \cite[Theorem 1.5]{A-CGV20} it is shown that $E(2,\mathrm{O}(2))\simeq S^2\vee S^2\vee S^3$, and it follows that $\pi_2(E(2,\mathrm{O}(2)))\cong \Z^2$, whereas $\pi_2(B\mathrm{SO}(2))=\Z$. However, we will see that $\pi_2(E(2,G))$ splits off $\pi_2(B[G,G])$. 
\end{remark}

As mentioned before, by \cite[Theorem 6.3]{AdemCohenTorres12}, the following corollary is an immediate consequence of Remark \ref{rem-pi2e2gconn}.

\begin{corollary}
Let $G$ be a connected compact Lie group. Then $\pi_2(B(2,G))\cong \pi_1(G)\oplus\pi_1([G,G])$.
\end{corollary}

\begin{remark}
Let $G$ be a semisimple simply-connected compact Lie group. Proposition \ref{prop-pi2E2g} then asserts that $\pi_2(E(2,G))= 0$,  and hence $\pi_2(B(2,G))$ is trivial, as well. As $\pi_3(BG)=0$, it follows that $H_3(B(2,G);\Z)\cong\pi_3(B(2,G))\cong \pi_3(E(2,G))$. By inspection of the spectral sequence $E_{p,q}^2=H_pH_q(C_\bullet(G))\Longrightarrow H_{p+q}(B(2,G))$, we can conclude that $\pi_3(E(2,G))\cong H_3(H_0(C_\bullet(G);\Z))$. For the classical simple Lie groups $\mathrm{SU}(n)$ and $\mathrm{Sp}(m)$, it is known that $\pi_3(E(2,G))=0$, as the spaces $C_k(G)$ are path-connected. But this is not known in general. For instance the space of commuting triples $C_3(\mathrm{Spin}(7))$ is not path-connected, so $\pi_3(E(2,\mathrm{Spin}(7)))$ could potentially be non-trivial.
\end{remark}

\section{$\pi_2(E(2,G))$ for extensions of a finite group by a torus}\label{Section extens}

In this section we will focus on studying compact Lie groups $G$ with abelian $G_0$. We can set up as follows. 

Let $G$ be an extension of a finite group $Q$ by a torus $T=(S^1)^k$. Pick an element $q\in Q$, and let $\bar{q}\in G$ be a lift of $q$ under the quotient map. Recall that we use the commutator convention $[x,y]=x^{-1}y^{-1}xy$. We can write the commutator $[\bar{q},-]\colon T\to T$ as a product of maps $[c_{\bar{q}}\circ(-)^{-1}]\cdot Id$, where $c_{\bar{q}}\colon T\to T$ is conjugation by $\bar{q}$ and $(-)^{-1}$ is inversion in $T$. As $T$ is abelian, conjugation is independent of the choice of a lift, hence 
\[\psi(q):=[\bar{q},-]\] 
is a well defined group homomorphism $T\to T$ for ever $q\in Q$. The induced homomorphism in $\pi_1$ is then
\begin{align*}
\psi(q)_*\colon&\Z^k\to \Z^k\\ 
&x\mapsto x - \pi_1(c_{\bar{q}})(x)
\end{align*}

\begin{lemma}\label{lemma-pi1-comm}
Let $G$ be an extension of a finite group $Q$ by a torus $T$. Let $[G,G]_0$ be the connected component of the identity of the commutator subgroup. Then 
\begin{enumerate}
\item Every element in $[G,G]_0$ is a product of the form $\psi(q_1)(t_1)\cdots\psi(q_n)(t_n)$, where $q_i\in Q$ and $t_i\in T$.
\item The subgroup $\sum_{q\in Q}\im\psi(q)_*\subset\pi_1([G,G]_0)$ is of maximal rank.
\end{enumerate}
\end{lemma}

\begin{proof}
Borel and Serre \cite[Lemma 5.1, footnote p. 152]{BorelSerre64} noticed that when $G$ is a compact Lie group that is not connected there exists a finite group $F\subset G$ such that $G=FG_0$. That is, we can choose the lifts $\bar{q}$ to be of finite order. As $G$ is compact, the elements in $[G,G]$ are finite products of commutators. We can use the identity $t\bar{q}=\bar{q}c_{\bar p}(t)$ to write any single commutator as
\begin{align} 
 [\bar{p}s,\bar{q}t] &= [\bar{p},\bar{q}]c_{[\bar{p},\bar{q}]}(s^{-1})c_{\bar{q}^{-1}\bar{p}\bar{q}}(t^{-1})c_{\bar{q}}(s)t\label{1}\\
 &= [\bar{p},\bar{q}]\psi([p,q])(s)\psi(q^{-1}pq)(t)\psi(q)(s^{-1})\label{2}\,,
 \end{align}
 where $p,q\in Q$ and $s,t\in T$. Any finite product of these can be rearranged similarly as in (\ref{1}) to an  expression of the form $[\bar{q}_1,\bar{p}_1]\cdots[\bar{q}_m,\bar{p}_m]c(s_1,t_1,...,s_m,t_m)$, where $c(s_1,t_1,...,s_m,t_m)$ is a product of conjugates of $s_i,t_i\in T$ by words in terms of lifts of $p_i,q_i\in Q$.  Now since $[G,G]_0\subset T$ and $c(s_1,t_1,...,s_m,t_m)$ can be path-connected to the identity we see that $r=[\bar{q}_1,\bar{p}_1]\cdots[\bar{q}_m,\bar{p}_m]\in T$ and has finite order. Thus we have an exhaustive description of the elements of $[G,G]_0$. Finally, similarly as in (\ref{2}), we see that they are all products of the form $r\psi(q_1)(t_1)\cdots\psi(q_n)(t_n)$ for some $q_i\in Q$ and $t_i\in T$.

  Now given a group decomposition of $[G,G]_0$ into $S^1$'s, we can choose $r\psi(q_1)(t_1)\cdots\psi(q_m)(t_m)\in S^1_j$ of infinite order, where $S^1_j$ is the $j$-th $S^1$-factor of $[G,G]_0$. Then as each $\psi(q_i)$ is a homomorphism, taking $N=\text{ord}(r)$, we have $\psi(q_1)(t_1^N)\cdots\psi(q_m)(t_m^N)\in S^1_j$, and since it is still of infinite order, $\psi(q_1)(-)\cdots\psi(q_m)(-)$ must cover $S^1_j$. Indeed, let $Y\subset T\times \cdots\times T$ be a circle subgroup, such that $(t_1^N,...,t_m^N)\in Y$ (for example take $Y$ as the closure of a 1-parameter subgroup through $(t_1^N,...,t_m^N)$). Then \[Y\xrightarrow{\psi(q_1)(-)\cdots\psi(q_m)(-)}S^1_j\] is a covering map, and part 1 follows.
 
 For part 2, consider the homomorphism \[\underbrace{T\times\cdots\times T}_{q\in Q}\xrightarrow{\prod_{q\in Q} \psi(q)} [G,G]_0\,,\]
 which by part 1 is surjective. Since it is a continuous homomorphism between products of $S^1$'s, and $\text{rank}([G,G]_0)\leq\text{rank}(T)$, the induced map in fundamental groups is of maximal rank, as desired.
\end{proof}

\begin{remark}
In general $\psi(q)_*$ may not be surjective. For example consider $G=O(2)$, where $T=SO(2)=[G,G]$, and $Q=\Z/2$. If $\tau\in \Z/2$ is a generator, then $[\overline{\tau},r_{\theta}]=r_{2\theta}$, where $r_\theta\in SO(2)$ is rotation by $\theta$. It follows that $\psi(\tau)_*\colon \Z\to \Z$ is multiplication by 2.
\end{remark}

\subsection{Cocycles and associated clutching function.}\label{subsec-cover}

Let $\mathcal{C}=\{C_1,C_2,C_3\}$ be the closed cover of $S^2$ given by $C_1=\{\vec{x}\in S^2:x_0\leq 0\}$, $C_2=\{\vec{x}\in S^2:x_0\geq0\text{ and }x_2\geq 0\}$ and $C_3=\{\vec{x}\in S^2:x_0\geq0\text{ and }x_2\leq 0\}$. Let $G$ be a topological group. A $G$-valued \emph{commutative cocycle} over $\mathcal{C}$ is a cocycle 
\[\alpha=\{\alpha_{12}\colon C_1\cap C_2\to G,\;\alpha_{13}\colon C_1\cap C_3\to G,\;\alpha_{23}\colon C_2\cap C_3\to G\}\]
satisfying the usual cocycle equation with the added condition that the commutators 
 \[[\alpha_{12}(x),\alpha_{23}(x)]=[\alpha_{12}(x),\alpha_{13}(x)]=[\alpha_{23}(x),\alpha_{13}(x)]=1\in G\]
  for any $x\in C_1\cap C_2\cap C_3$ (of course one can generalize this definition to any cover). Note that each non-degenerate double intersection is a semicircle and the triple intersection consists of 2 points --- `front’ $(0,1,0$) and `back' $(0,-1,0)$.  

 By definition, a commutative cocycle $\alpha$ will give rise to a simplicial map after applying the nerve functor $N\alpha\colon N(\mathcal{C})\to B_\bullet G$, where $N(\mathcal{C})$ is the \u{C}ech complex as considered in \cite{DuggerIsaksen}. It will factor through the inclusion $i_\bullet \colon B_\bullet(2,G)\to B_\bullet G$. That is, we obtain a map 
\[f_\alpha\colon S^2\xrightarrow{\simeq}|N(\mathcal{C})|\xrightarrow{|N(\alpha)|}B(2,G)\]
Note that when $\alpha$ is commutative we can point-wise invert $\alpha$ in $G$ and still get a cocycle which we may call the \emph{inverse cocycle} and will denote it by $\alpha^{-1}$. Then we have the identity $\phi^{-1}f_\alpha=f_{\alpha^{-1}}$.

To a $G$-valued cocycle $\alpha$ over $\mathcal{C}$ one can associate a $G$-bundle $E_\alpha\to S^2$. By \cite[Lemma 3.3]{RamrasVillarreal19} $if_\alpha$ classifies $E_\alpha$. Moreover, by \cite[Lemma 3.2]{RamrasVillarreal19} $E_\alpha$ is clutched by the map $\varphi_{\alpha}\colon S^1\to G$ given by
\[
\varphi_\alpha(x)=\left\{
\begin{matrix}
\alpha_{12}(x)\alpha_{23}(x)&\text{if }x\in C_1\cap C_2\\
\alpha_{13}(x)&\text{if }x\in C_1\cap C_3
\end{matrix}
\right.
\]
(where we are identifying $(C_1\cap C_2)\cup (C_1\cap C_3)$ with $S^1$). We refer to $\varphi_\alpha$ as \emph{the clutching function induced by} $\alpha$. 

As both $\varphi_\alpha$ and $if_\alpha$ classify the same bundle there is a relation between their homotopy classes. When considered as pointed maps, this relation is attained by the natural isomorphism $\pi_2(BG)\cong \pi_1(G,1_G)\cong \pi_1(G,\varphi_\alpha(1_{S^1}))$.

\subsection{The commutator map via commutative cocycles}

For a brief moment let us assume $G$ is any Lie group. Recall the simplicial model of $E(2,G)$ whose $n$-simplices are $E_n(2,G)=\{(g_0,...,g_n): \langle g_0^{-1}g_1,...,g_{n-1}^{-1}g_n\rangle\;\text{is abelian}\}$ described in Section \ref{sec-pre}. Such an $(n+1)$-tuple in $E_n(2,G)$ is called affinely commutative. This model is particularly useful because it allows us to construct \emph{the commutator map} which has proven to be a powerful tool to study homotopical properties of $E(2,G)$. Precisely, the commutator map $\mathfrak{c}$ is defined simplicially as
\begin{align*}
E_k(2,G)&\xrightarrow{\mathfrak{c}_k} B_k[G,G]\\
(g_0,...,g_k)&\mapsto ([g_0,g_1],...,[g_{k-1},g_k])  
\end{align*}
for every $k\geq 1$, and $\mathfrak{c}:=|\mathfrak{c}_\bullet|\colon E(2,G)\to B[G,G]$ (see \cite[Section 3]{AntolinGritschacherVillarreal21} for more details).

\begin{proposition}\label{prop-max-rank}
Let $G$ be an extension of a finite group $Q$ by a torus $T$. Then the image of the commutator map \[\mathfrak{c}_*\colon \pi_2(E(2,G))\to \pi_2(B[G,G])\] 
is of maximal rank.
\end{proposition}

\begin{proof}
Consider the cover $\mathcal{C}=\{C_1,C_2,C_3\}$ of $S^2$ as in Subsection \ref{subsec-cover}. We parametrize the semi-circles $C_{r}\cap C_s$ by $C_{rs}(t)$ where $t\in [0,1]$. For every $q\in Q$ and every loop $x\colon[0,1]\to T$ based at the identity $1\in T$, let us now define the 0-cocycles $_{q,x}\varphi_{r}\colon C_{r}\to G$ by 
\begin{align*}
_{q,x}\varphi_{1}=& \;h\\
_{q,x}\varphi_{2}=& \;\bar{q}\\
_{q,x}\varphi_{3}=& \;1\,,
\end{align*}
where $h$ is an extension to $C_1$ obtained from the null map $x*\overline{x}$ which applies $x$ over $C_{12}(t)$ and $x$ in the opposite direction over $C_{13}(t)$. We claim that these maps induce a simplicial map ${_{q,x}\varphi}_\bullet\colon N(\mathcal{C})\to E_\bullet(2,G)$. To see this, we need to analyse ${_{q,x}\varphi_{i}^{-1}}_{q,x}\varphi_{j}$, with $i<j$. On double intersections we obtain
\begin{align*}
({_{q,x}\varphi_{1}^{-1}}_{q,x}\varphi_{2})(C_{12}(t))=& \;x(t)^{-1}\bar{q}\\
({_{q,x}\varphi_{2}^{-1}}_{q,x}\varphi_{3})(C_{23}(t))=& \;\bar{q}^{\;-1}\,,
\end{align*} 
and on triple intersections we get $\bar{q}$ and $\bar{q}^{\;-1}$. Thus $({_{q,x}\varphi_{1}},
{_{q,x}\varphi_{2}},{_{q,x}\varphi_{3}})$ is affinely commutative on triple intersection points. Therefore the maps ${_{q,x}\varphi_r}$ induce a simplicial map ${_{q,x}\varphi}_\bullet$ as claimed. Consider the composite 
\[N(\mathcal{C})\xrightarrow{{_{q,x}\varphi}_\bullet}E_\bullet(2,G) \xrightarrow{\mathfrak{c}_\bullet}B_\bullet[G,G]\,,\] 
where $\mathfrak{c}_\bullet$ is the simplicial model of the commutator map. It can be seen that the composite is induced by the cocycle over $\mathcal{C}$ that is constant over $C_1\cap C_3$ and $C_2\cap C_3$ with value 1, and $t\mapsto [x(t),\bar{q}]$ over the parametrized intersection $C_{12}^\prime(t)$, where $t\in [0,1]$. One can verify that after taking geometric realization $\mathfrak{c}\circ{_{q,x}\varphi}\colon S^2\to B[G,G]$ classifies the same bundle as the clutching map
\[S^1\xrightarrow{x}T\xrightarrow{\psi(q)}[G,G]_0\xrightarrow{(-)^{-1}}[G,G]_0\,,\]
for example by \cite[Lemma 3.3]{RamrasVillarreal19}. The proposition now follows from Lemma \ref{lemma-pi1-comm} part 2.
\end{proof}

The space $B(2,G)$ is equipped with an involution $\phi^{-1}\colon B(2,G)\to B(2,G)$ given by the realization of the maps $\phi^{-1}_k\colon C_k(G)\to C_k(G)$ that invert each coordinate. This involution is closely related to the commutator map (see \cite[Remark 11]{AntolinGritschacherVillarreal21}), which we will need to prove the following.  

\begin{proposition}\label{prop-single-comm}
Let $G$ be an extension of a finite group $Q$ by a torus $T$, in which $[G,G]_0$ consists of single commutators. Then $\mathfrak{c}_*\colon \pi_2(E(2,G))\to \pi_2(B[G,G])$ is surjective.
\end{proposition}
\begin{proof}
Choose a group decomposition of $[G,G]_0$ into a product of $S^1$'s. Similarly as in the proof of Proposition \ref{prop-max-rank}, we can write any single commutator as a product of two $\psi(q)$'s. Thus, for each $i$-th projection $\mathrm{proj}_i\colon[G,G]_0\to S^1$ there is a product of commutators $\psi(q_i)(-)\psi(p_i)(-)$ and some circle subgroup $S^1_i\subset T\times T$ such that the image $\psi(q_i)(-)\psi(p_i)(-)(S^1_i)= \mathrm{proj}_i([G,G]_0)$. In particular the composition
\begin{align}
S^1_i\hookrightarrow T\times T\xrightarrow{\psi(q_i)(-)\psi(p_i)(-)} [G,G]_0\xrightarrow{\mathrm{proj}_i} S^1\label{deg n map}
\end{align}
yields a degree $n$ map for some $n\in\Z$, as it is a non-trivial continuous endomorphism of $S^1$. By post composing (\ref{deg n map}) with $(-)^{-1}$ if necessary, we may assume $n>0$. Then by an appropriate parametrization of $S^1\subset T\times T$, say $z\mapsto (x(z),y(z))$, the image of the $n$-th root of unity $\xi_n=e^{\frac{\pi i}{n}}$ will satisfy $[px(\xi_n),qy(\xi_n)]=1$. Parametrize $C_r\cap C_s$ by $C_{rs}(t)$, $t\in [0,1/n]$. Then define $\alpha$ as
\begin{align*}
\alpha_{12}(C_{12}(t))=& \;px(t)qy(t)\\
\alpha_{23}(C_{23}(t))=& \;(qy(t))^{-1}\\
\alpha_{13}(C_{13}(t))=& \;px(t)
 \end{align*}
which is a commutative cocycle by construction. We can represent the associated clutching functions as in the following pictures:
\begin{align}
\begin{tikzpicture}
\draw{(-4.7,0)} node{$t=\frac{1}{n}$};
\draw{(-5,-1)} node{$\varphi_{\alpha}$};
\draw(-3,0) circle(1cm);
\draw(-4,0) node {$-$};
\draw(-2,0) node {$-$};
\draw(-1.3,0) node {$t=0$};
\draw[blue,thick,->](-2,0) arc (0 : 40 : 1);
\draw[blue,thick,->](-2,0) arc (0 : -40 : 1);
\draw(-2.9,1.5) node {$px(t)$};
\draw(-2.9,-1.5) node {$px(t)$};
\draw(1.3,0) node {$t=\frac{1}{n}$};
\draw(1,-1) node {$\varphi_{\alpha^{-1}}$};
\draw(3,0) circle (1cm);
\draw(4,0) node {$-$};
\draw(2,0) node {$-$};
\draw(4.7,0) node {$t=0$};
\draw[blue,thick,->](4,0) arc (0 : 40 : 1);
\draw[blue,thick,->](4,0) arc (0 : -40 : 1);
\draw(3.2,1.5) node {$(qy(t))^{-1}(px(t))^{-1}qy(t)$};
\draw(3.2,-1.5) node {$(px(t))^{-1}$};
\end{tikzpicture}\label{pic:cf}
\end{align}
It follows that $\varphi_\alpha$ is null-homotopic, and $\varphi_{\alpha^{-1}}$ is homotopic to (\ref{deg n map}). Now since $f_{\alpha}\colon S^2\to B(2,G)$ is null, it induces a class in $\pi_2(E(2,G))$. 

Consider the diagram
\[
\xymatrix{
E(2,G)\ar[rrr]^{\mathfrak{c}}\ar@{=}[d]&&&B[G,G]\ar[d]^j\\
E(2,G)\ar[r]&B(2,G)\ar[r]^{\phi^{-1}}&B(2,G)\ar[r]^{i}&BG
\,,}
\]
which commutes up to homotopy (\cite[Remark 11]{AntolinGritschacherVillarreal21}). 

Since $T/[G,G]_0$, $[G,G]_0$ and $T$ are tori, the inclusion $[G,G]_0\to T$ induces a splitting $\pi_1(T)\cong \pi_1([G,G]_0)\oplus\pi_1(T/[G,G]_0)$.
We claim that $\pi_2(i\phi^{-1})(f_{\alpha})=\pi_2(if_{\alpha^{-1}})$ lies in $j_*(\pi_2(B[G,G]))$. Let $\varphi_{\alpha^{-1}}\colon S^1\to G$ be the associated clutching map of $\alpha^{-1}$ depicted in (\ref{pic:cf}). By construction $(-)^n\circ\varphi_{\alpha^{-1}}$ is homotopic to (\ref{deg n map}), thus $n[\varphi_{\alpha^{-1}}]\in \pi_1([G,G]_0)$, which implies $[\varphi_{\alpha^{-1}}]\in\pi_1([G,G]_0)$ as $\pi_1([G,G]_0)$ is a direct summand and $\pi_1(T/[G,G]_0)$ is torsion free. Our claim now follows.

We conclude that the composite $\pi_2(E(2,G))\xrightarrow{\mathfrak{c}_*} \pi_2(B[G,G])\xrightarrow{\mathrm{proj}_i} \pi_2(BS_i^1)$ is surjective for every $i$-th projection, so $\mathfrak{c}_*$ is surjective, as well. 
\end{proof}

\begin{example}\label{examp-exten-by-S1}
A generic example of an extension where $[G,G]_0$ consists of single commutators is when $G$ is an extension by $T=S^1$ (e.g. $G=O(2)$). Recall that $\mathrm{Aut}(S^1)=\Z/2$, then since $T$ is normal in $G$, either $T$ is fixed by any lift of $\pi_0(G)$ in $G$, in which case $T\subset Z(G)$ and $[G,G]_0$ is finite, or there exists a lift $a$ that acts by inversion on $T$. In the latter case we can write $z\in T$ as $[\sqrt{z^{-1}},a]$, and $[G,G]_0=T$. 
\end{example}

\begin{example}\label{prop-pi2cm-surj-conn}

Let $G$ be a semisimple connected compact Lie group, and $T$ a maximal torus. It is known that the normalizer $N_G(T)$ satisfies the hypothesis of Proposition \ref{prop-single-comm}. To see this, let $\text{Lie}(T)$ be the Lie algebra of $T$. Choose a basis of $\text{Lie}(T)^*$ consisting of simple roots $\{\alpha_i\}$. To each $\alpha_i$ there is an associated coroot $\alpha^\dagger_i$ in $\text{Lie}(T)$. These elements induce a Lie group homomorphism
 \[\exp(\widehat{\alpha}_i\colon \text{span}(\alpha^\dagger_i)\hookrightarrow \text{Lie}(T))\colon  S^1\to T\] 
 such that $\prod_i \exp(\widehat{\alpha}_i)\colon (S^1)^m\to T $ is an isomorphism, where $m$ is the rank of $T$. The Weyl group of $G$ is then generated by the reflections $s_i$ through the hyperplane $\ker(\alpha_i)$ which is also orthogonal to $\alpha_i^\dagger$ for all $1\leq i\leq m$. Therefore for each $i$ the action of $s_i$ on $\exp(\widehat{\alpha}_i)(S^1)$ is given by inversion, and as in Example \ref{examp-exten-by-S1} we can now write any element in $T$ as a single commutator. In particular $T=[N_G(T),N_G(T)]_0$. This is a way of proving that any element in a compact connected semisimple Lie group consists of single commutators \cite{Goto49}.
\end{example}

We finish this subsection with an application concerning the splitting of the commutator map after looping \cite[Question 21]{AntolinGritschacherVillarreal21}.

\begin{corollary}\label{cor-omegacsplit}
Let $G$ be an extension of a finite group by a torus, where $[G,G]_0$ consists of single commutators. Then $\Omega\mathfrak{c}\colon \Omega E(2,G)\to [G,G]$ splits, up to homotopy.
\end{corollary}
\begin{proof}
The commutator subgroup is a disjoint union of copies of $[G,G]_0$ indexed by representatives $a$ of each element in $\pi_0[G,G]$, that is, 
\[[G,G]=\bigsqcup_{[a]\in \pi_0[G,G]}a[G,G]_0\,.\]
 
To define a splitting $s_1$ at the component of identity we only need to specify maps $s^i_1\colon S^1\to \Omega_0 E(2,G)$ that composed with 
\[\Omega_0 E(2,G)\xrightarrow{\Omega\mathfrak{c}} [G,G]_0\xrightarrow{\mathrm{proj}_i}S^1 \]
are the identity, up to homotopy. But by Proposition \ref{prop-single-comm}, $\pi_1(\Omega\mathfrak{c})$ is surjective, hence we have a section $\sigma_1$ (which in fact can be chosen as a group homomorphism since $\pi_1([G,G]_0)$ is free abelian). Thus we can define $s^i_1:=\sigma_1[\iota_i]$ for each $i$-th generator $S^1\xrightarrow{\iota_i} [G,G]_0$ in $\pi_1([G,G]_0)$.

To extend $s_0$ to $\Omega E(2,G)$ we use \cite[Proposition 9]{AntolinGritschacherVillarreal21} which asserts that $\pi_0(\Omega\mathfrak{c})$ has a section $\sigma$. Then we extend on each $a[G,G]$ by picking an element $\sigma_0(a)\in\sigma[a]$ for every $[a]\in \pi_0[G,G]$. Then we can set $s=\sigma_0s_0$.
\end{proof}

\begin{remark}
When $G$ is an extension of a finite group by a torus, where the elements in $[G,G]_0$ are not necessarily single commutators, $[G,G]$ still splits off the loop space $\Omega E(2,G)$, up to homotopy. It is mostly a consequence of Proposition \ref{prop-max-rank} and the fact that the commutator map $\mathfrak{c}$ is surjective in $\pi_1$, albeit the splitting may not be attained by the looped commutator map as in Corollary \ref{cor-omegacsplit}. 
\end{remark}

\section{Proof of Theorem \ref{MainTheorem1} and applications}\label{SectionProofMainTh}

We have all the ingredients to prove our main result, but we still have to pin down generators of $\pi_2(B[G,G])$. For that we require the following lemma.

\begin{lemma}\label{lem-comm-dec-as-cenprod}
Let $G$ be a compact Lie group. Then there exists a finite subgroup $F\subset G$, such that $[G,G]_0$ is isomorphic to the central product
\[[G,G]_0\cong[F,Z(G_0)_0]\times _Z [G_0,G_0]\,,\] 
where $Z=[F,Z(G_0)_0]\cap[G_0,G_0]$. In particular there is a split short exact sequence \[0\to \pi_1([G_0,G_0])\to \pi_1([G,G]_0)\to\pi_1([F,Z(G_0)_0]/Z)\to 0\]
induced by the inclusion $[G_0,G_0]\to [G,G]_0$ and the projection $[G,G]_0\to [F,Z(G_0)_0]/Z$. 
\end{lemma}
\begin{proof}
As in the proof of Lemma \ref{lemma-pi1-comm} part 1, we can find a finite group $F\subset G$ such that $G=FG_0$, and decomposition \cite[Theorem 4.29]{Knapp} then yields $G=FZ(G_0)_0[G_0,G_0]$. Since $[G_0,G_0]$ is normal in $G$ we obtain that 
\[[G,G]=[F,F][F,Z(G_0)_0][G_0,G_0]\,.\]
As $Z(G_0)_0$ is a characteristic subgroup of $G_0$ it is normal in $G$, that is, $[F,Z(G_0)_0]\subset Z(G_0)_0$ and since it is closed and connected, it is a torus. Moreover, since $[G_0,G_0]\subset[G,G]_0$ is a connected subgroup, as well, $[G,G]_0=[F,Z(G_0)_0][G_0,G_0]$. Then as in the proof of Lemma \ref{lem-e2semisimple} we have an isomorphism $[G,G]_0\cong[F,Z(G_0)_0]\times _Z[G_0,G_0]$, where $Z\subset Z[G_0,G_0]$. In particular we have a short exact sequence $0\to\pi_1([G_0,G_0])\to \pi_1([G,G]_0)\to \pi_1([F,Z(G_0)_0]/Z)\to 0$ obtained from the fibration sequence $[G_0,G_0]\to [F,Z(G_0)_0]\times _Z[G_0,G_0]\to [F,Z(G_0)_0]/Z$. As $\pi_1([F,Z(G_0)_0]/Z)$ is a free abelian group and $\pi_1([G,G]_0)$ is abelian, the sequence splits.
\end{proof}

Thus we can study separately the elements in $\pi_1([G_0,G_0])$ which are all torsion (as $[G_0,G_0]$ is semisimple), and the elements in $\pi_1([F,Z(G_0)_0]/Z)$ which is a finitely generated free abelian group. In both cases these elements arise from extensions of finite groups by tori.

\begin{proof}[Proof of Theorem \ref{MainTheorem1}]

We divide the proof of the theorem into 3 claims. 

\begin{claim}\label{claimisoconn} 
$\pi_2(\mathfrak{c})$ is an isomorphism when $G$ is connected. 
\end{claim}

\begin{proof}

The existence of an isomorphism $\pi_2(E(2,G))\cong\pi_1([G,G])$ has already been showed in Proposition \ref{prop-pi2E2g}, thus to conclude $\pi_2(\mathfrak{c})$ is an isomorphism, we only need to show it is surjective, as in this situation $[G,G]$ is semisimple, so $\pi_1([G,G])$ is finite. We have the following commutative diagram 
\[
\xymatrix{
E(2,[G,G])\ar[d]_{\simeq}\ar[r]^{\cm}&B[G,G]\ar@{=}[d]\\
E(2,G)\ar[r]^{\cm}&B[G,G]\,,
}
\]
where the left vertical map is a homotopy equivalence by Lemma \ref{lem-e2centralproduct}.  So we can further assume that $G$ is semisimple.

Recall that if $T$ is a maximal torus of $G$, the flag manifold $G/T$ is simply-connected (as it has a CW structure with only even dimensional cells). Then the long exact sequence of homotopy groups for the fiber sequence $G/T\to BT\to BG$ gives a surjection $\pi_2(BT)\to\pi_2(BG)$. 

Let $N=N_G(T)$. We have the following commutative diagram:
\[
\xymatrix{
E(2,N)\ar[d]\ar[r]^{\cm}&B[N,N]\ar[d]\\
E(2,G)\ar[r]^{\cm}&BG\,.
}
\]

As explained in Example \ref{prop-pi2cm-surj-conn}, when $G$ is semisimple, $T=[N,N]_0$ and $T$ consists of single commutators. We can apply Proposition \ref{prop-single-comm} to see that the top horizontal map is surjective on $\pi_2$. Since the inclusion $BT\to BG$ factors through $B[N,N]\to BG$, the right vertical arrow is surjective in $\pi_2$, implying the bottom horizontal arrow is surjective in $\pi_2$, as well. Then by Proposition \ref{prop-pi2E2g}, $\pi_2(\mathfrak{c})$ must be an isomorphism, as $\pi_1(G)$ is a finite group when $G$ is semisimple. 
\end{proof}

\begin{claim}\label{claimpi2csurjtorsion}
$\pi_2(\mathfrak{c})$ is a split surjection onto the torsion part of $\pi_2(B[G,G])$.
\end{claim}

\begin{proof}
Consider the inclusion $G_0\to G$ and the commutative diagram
  \[
\xymatrix{\pi_2(E(2,G_0))\ar[r]^{\pi_2(\mathfrak{c})}\ar[d]&\pi_2(B[G_0,G_0])\ar[d]\\
\pi_2(E(2,G))\ar[r]^{\pi_2(\mathfrak{c})}&\pi_2(B[G,G])\,.}
\]
By Lemma \ref{lem-comm-dec-as-cenprod}, the right vertical arrow is an isomorphism onto the torsion part of $\pi_2(B[G,G])$. By Claim 1, the top horizontal arrow is an isomorphism. Hence the bottom $\pi_2(\mathfrak{c})$ is split surjective onto the torsion part, as claimed.
\end{proof}

\begin{claim}\label{claimmaxrank}
The image of $\pi_2(\mathfrak{c})$ is of maximal rank.
\end{claim}

\begin{proof}

Let $F\subset G$ be a finite subgroup as in Lemma \ref{lem-comm-dec-as-cenprod}. 
It is not hard to see that $[F,Z(G_0)_0]/Z=[F,Z(G_0)_0/Z]$. This group can be analysed via the extension 
\[1\to Z(G_0)_0/Z\to FZ(G_0)_0/Z\to F/F\cap Z(G_0)_0\to 1\,.\]

We can identify $FZ(G_0)_0/Z$ with the quotient $(F\ltimes Z(G_0)_0/Z)/\varphi(F\cap G_0)$, where $\varphi\colon F\ltimes G_0\to F\ltimes Z(G_0)_0/Z$ is induced by the projection $G_0\to Z(G_0)_0/Z$. Since $G=F\ltimes G_0/F\cap G_0$, $\varphi$ descends to the quotient and we have a well defined homomorphism $\bar\varphi\colon G\to FZ(G_0)_0/Z$ that restricts to the projection $G_0\to Z(G_0)_0/Z$. Consider the commutative diagram
  \[
\xymatrix{\pi_2(E(2,G))\ar[d]_{E(2,\bar\varphi)_*}\ar[r]^{\pi_2(\mathfrak{c})}&\pi_2(B[G,G])\ar[d]^{\bar\varphi|_*} \\
\pi_2(E(2,FZ(G_0)_0/Z))\ar[r]^{\pi_2(\mathfrak{c})}&\pi_2(B[F,Z(G_0)_0/Z])\,.}
 \]
By construction $\bar\varphi|$ restricts to the projection $[G,G]_0\to [F,Z(G_0)_0/Z]$, and by Lemma \ref{lem-comm-dec-as-cenprod}, $\bar\varphi|_*$ is an isomorphism restricted to the torsion free part of $\pi_2(B[G,G])$. By Proposition \ref{prop-max-rank} the bottom horizontal arrow is of maximal rank, implying the top horizontal arrow is of maximal rank, as well. 
\end{proof}

Now we complete the proof of the theorem. Claim \ref{claimisoconn} is precisely part 2 of the theorem. By claims \ref{claimpi2csurjtorsion} and \ref{claimmaxrank}, the image of $\pi_2(\mathfrak{c})$ is a subgroup of the form $A\oplus\pi_2(B[G_0,G_0])\subset \pi_2(B[G,G])$, where $A$ is a maximal rank free abelian subgroup, and $\pi_2(B[G_0,G_0])$ is the torsion part of $\pi_2(B[G,G])$. In particular $A\oplus\pi_2(B[G_0,G_0])$ is isomorphic to $\pi_2(B[G,G])$. Combining Claim 2 and the fact that $A$ is free abelian, we can guarantee the existence of a split surjective homomorphism $\pi_2(E(2,G))\to A\oplus \pi_2(B[G_0,G_0])$, and part 1 of the theorem now follows.
%
\end{proof}

We record some applications of Theorem \ref{MainTheorem1}. First we prove Corollary \ref{cor-e2g2conn}.

\begin{proof}[Proof of Corollary \ref{cor-e2g2conn}]
Suppose $E(2,G)$ is 2-connected. In \cite[Proposition 9]{AntolinGritschacherVillarreal21} it is shown that $\pi_2(\mathfrak{c})$ is surjective, thus if $\pi_1(E(2,G))=1$, $[G,G]$ is connected. Part 1 of Theorem \ref{MainTheorem1} implies that $\pi_2(B[G,G])=\pi_1([G,G])=0$. 
\end{proof}

An equivalent formulation of a transitionally commutative (tc) structure on a principal $G$-bundle, as defined in \cite{GritschacherTh}, is that its classifying map is a factorization, up to homotopy, through the inclusion $i\colon B(2,G)\to BG$. Consider the involution $\phi^{-1}\colon B(2,G)\to B(2,G)$. Then any principal $G$-bundle with a tc structure has an associated involution. Let us denote $p\colon E(2,G)\to B(2,G)$ the pullback of $EG\to BG$ along $i$.  

\begin{lemma}\label{lem-tc-via-comap}
Let $G$ be a well based topological group and suppose $E$ is a principal $G$-bundle over a space $X$ with a reduction to the commutator subgroup of $G$. If its classifying map $f\colon X\to B[G,G]$ factors through the commutator map $\mathfrak{c}$, up to homotopy, then $E$ has a transitionally commutative structure whose involution is a trivial bundle.  
\end{lemma}
\begin{proof}
Consider the diagram
\[
\xymatrix{&E(2,G)\ar[d]^{\mathfrak{c}}\ar[r]^{\phi^{-1}p}&B(2,G)\ar[d]^{i}\\
X\ar[r]_{f}\ar[ru]^{f^\prime}&B[G,G]\ar[r]_{j}&BG\,.}
\]
By \cite[Remark 11]{AntolinGritschacherVillarreal21} the square on the right commutes up to homotopy. The composite $\phi^{-1}pf^\prime$ is a tc structure with associated involution $pf^\prime$ which after composing with $i$ is null-homotopic.  
\end{proof}

Part 2 of Theorem \ref{MainTheorem1} in particular implies that when $G$ is connected, the homotopy fiber of $\mathfrak{c}$ is 2-connected. A standard obstruction theory argument and Lemma \ref{lem-tc-via-comap} show the following: 

\begin{corollary}\label{cor-tcstdim3}
Let $X$ be a CW complex of dimension $\leq 3$, and let $G$ be a connected compact Lie group. Then every principal $G$-bundle over $X$ with a reduction to the commutator subgroup has a transitional commutative structure whose involution is a trivial bundle.
\end{corollary}

For example, any oriented vector bundle over a CW complex of dimension $\leq 3$, will posses a tc structure.

{\footnotesize {\sc
  Instituto de Matem\'aticas, UNAM, Mexico City, Mexico}\\ 
  \emph{E-mail address}:
  \href{mailto:villarreal@matem.unam.mx}{\texttt{villarreal@matem.unam.mx}}}

\end{document}